\documentclass{amsart}

\usepackage{colonequals}
\usepackage[alphabetic]{amsrefs}
\usepackage{enumitem}
\usepackage{amssymb}

\newtheorem{theorem}{Theorem}[section]
\newtheorem{lemma}[theorem]{Lemma}
\newtheorem{proposition}[theorem]{Proposition}

\theoremstyle{definition}

\newtheorem{example}[theorem]{Example}
\newtheorem{question}[theorem]{Question}
\newtheorem{conjecture}[theorem]{Conjecture}

\theoremstyle{remark}
\newtheorem{remark}[theorem]{Remark}

\numberwithin{equation}{section}

\newcommand{\bP}{\operatorname{\mathbb{P}}}
\newcommand{\F}{\operatorname{\mathbb{F}}}
\newcommand{\Q}{\operatorname{\mathbb{Q}}}
\newcommand{\disc}{\operatorname{disc}}
\newcommand{\Gal}{\operatorname{Gal}}

\makeatletter
\@namedef{subjclassname@2020}{%
  \textup{2020} Mathematics Subject Classification}
  \makeatother

\begin{document}

\title{A Bertini type theorem for pencils over finite fields}

\author{Shamil Asgarli}
\address{Department of Mathematics, University of British Columbia, Vancouver, BC V6T 1Z2}
\email{sasgarli@math.ubc.ca}

\author{Dragos Ghioca}
\address{Department of Mathematics, University of British Columbia, Vancouver, BC V6T 1Z2}
\email{dghioca@math.ubc.ca}


\begin{abstract}
We study the question of finding smooth hyperplane sections to a pencil of hypersurfaces over finite fields. 
\end{abstract}

\maketitle

\section{Introduction}

Given a smooth projective variety $X\subset \bP^n$ over the complex numbers, the classical Bertini theorem asserts the existence of a hyperplane $H$ such that $X\cap H$ is smooth. The statement remains valid over an arbitrary infinite field $k$. For example, every smooth $\Q$-variety admits a smooth $\Q$-hyperplane section. However, if 
$k=\F_{q}$ is a finite field, there are counter-examples to the statement. The following example is due to Nick Katz \cite{Kat99}. Consider the surface $S\subset\bP^3_{\F_q}$ defined by
\[
X^q Y -  X Y^q + Z^q W - Z W^q = 0
\]
One can check that each $\F_q$-hyperplane $H\subset\bP^3$ is tangent to the surface $S$, and so $S\cap H$ is singular for every choice of $H$ in this case \cite{ADL19}*{Example 3.4}. 

If the field $\F_q$ has sufficiently large cardinality with respect to the degree of $X$, then we still expect to find smooth hyperplane sections. A theorem of Ballico \cite{Bal03} shows that for $q\geq d(d-1)^{n-1}$, any smooth hypersurface $X\subset \bP^n$ of degree $d$ admits an $\F_q$-hyperplane $H$ such that $X\cap H$ is smooth. When $X$ is a plane curve, a sharper bound of $q\geq d-1$ has been obtained under a stronger hypothesis of reflexivity \cite{Asg19}. 

We restrict our attention to the case of hypersurfaces. If $X\subset\bP^n$ is a hypersurface, we say that a given hyperplane $H$ is \emph{transverse} to $X$ if $X\cap H$ is smooth. 

In this paper, we study a pencil of hypersurfaces defined over $\F_q$ and ask for an $\F_q$-hyperplane which is simultaneously transverse to all the $\F_q$-members of the pencil. We take two different hypersurfaces $X_1 = \{F=0\}$ and $X_2=\{G=0\}$ of the same degree, and consider the $\F_q$-members of the pencil generated by $X_1$ and $X_2$. In other words, we examine the $q+1$ hypersurfaces,
$$
X_{[s:t]} = \{s F + t G = 0\}
$$ 
where $[s:t]\in\bP^1(\F_q)$. The main question can be phrased as follows:

\begin{question}
\label{quest:1}
Suppose that each member of the pencil spanned by $X_1$ and $X_2$ admits a transverse hyperplane over $\overline{\F_q}$. Provided that $q$ is sufficiently large with respect to $d$, can we find an $\F_q$-hyperplane $H$ such that $H$ is simultaneously transverse to $X_{[s:t]}$ for each $[s:t]\in\bP^1(\F_q)$?
\end{question}

The case $d=1$ is clear, because we can simply pick $H$ to be any hyperplane that is not in the pencil, and any two distinct hyperplanes intersect transversely. We assume $d>1$ throughout the paper. In a similar vein with Question~\ref{quest:1}, one may be inclined to ask for the existence of an $\overline{\F_q}$-hyperplane $H$ such that $H$ is transverse to all the $\overline{\F_q}$-members of a given pencil. However, this cannot be attained because any hyperplane $H$ must intersect some members of the pencil non-transversely. This is proved in Lemma~\ref{lem:degree}. 

Our main result asserts that the answer to Question~\ref{quest:1} is positive if we allow a base extension. The result rests on the following natural assumption on the pencil:

\textbf{Assumption on the pencil.} Suppose that $X_1, X_2\subset\bP^n$ are two hypersurfaces of degree $d$ defined over a finite field $k$. We will say that the pencil generated by $X_1$ and $X_2$ satisfies the condition \textbf{(T)} if the following hold:
\begin{enumerate}
    \item Each member of the pencil has a transverse hyperplane over $\overline{k}$.
    \item The pencil has a smooth member defined over $\overline{k}$.
\end{enumerate}

\begin{theorem}\label{main-theorem}
Let $n\geq 2$ and $d \geq 2$ be positive integers with $p \nmid n(d-1)$. Suppose that $X_1, X_2\subset\bP^n$ are two hypersurfaces of degree $d$ defined over a finite field $k$ of characteristic $p$ satisfying the assumption (\textbf{T}). Then there exists a finite field extension $k'/k$ such that the following holds: for all finite fields $\F_q\supseteq k'$, there exists an $\F_{q}$-hyperplane $H$ such that $H$ is transverse to $X_{[s,t]}$ for each $[s:t]\in\bP^1(\F_{q})$. 
\end{theorem}

\begin{remark} The finite field extension $k'/k$ depends only on $n$ and $d$, but not on the pencil itself. This assertion will be explicitly justified in the the proof.
\end{remark}

\begin{remark} As it will be mentioned in the proof, the hypothesis $p\nmid n(d-1)$ is needed to ensure that a certain map is separable. The required separability condition would also follow if we had instead imposed the following geometric condition: there exists a hyperplane $H$ defined over $\overline{k}$ such that $H$ is tangent to $n(d-1)^{n-1}$ many \emph{distinct} hypersurfaces in the pencil (see Lemma~\ref{lem:degree} for more context).
\end{remark}

\begin{remark} The hypothesis that a pencil has at least one smooth member defined over $\overline{k}$ is fairly mild. Indeed, a pencil can be viewed as a $\bP^1$ inside the parameter space of all hypersurfaces of degree $d$ in $\bP^n$. The condition that the pencil admits a smooth member is equivalent to the statement that the corresponding $\bP^1$ is not contained inside the discriminant hypersurface $\mathcal{D}_{d, n}$, which parametrizes singular hypersurfaces of degree $d$ in $\bP^n$. A generically chosen line is not contained inside $\mathcal{D}_{d, n}$, and so a generic pencil contains a smooth member.
\end{remark}

\begin{remark} According to our definition, a hyperplane $H$ is said to be transverse to $X$ if $H$ provides a smooth hyperplane section of $X$. This condition automatically implies that $H\notin X^{\ast}$ where $X^{\ast}$ is the dual hypersurface parametrizing tangent hyperplanes to $X$. More precisely, $X^{\ast}$ is the closure of the image of the Gauss map of $X$. However, the converse implication is not true. For example a line $L$ passing through the singularity of an irreducible nodal cubic $C$ is not transverse according to our definition, but still satisfies $L\notin C^{\ast}$. Some authors, such as \cite{Bal03}, defines $H$ to be transverse when the weaker condition $H\notin X^{\ast}$ is satisfied. Note that if $X$ is smooth, then $H\notin X^{\ast}$ if and only if $X\cap H$ is smooth. Thus, for smooth hypersurfaces, these two definitions of ``transverse hyperplane" coincide.
\end{remark}

We sketch here the plan for our paper. In Section~\ref{sec:plane} we discuss our Question~\ref{quest:1} in the context of plane curves. Then we prove Theorem~\ref{main-theorem} in Section~\ref{sec:proofs}. Finally,  we conclude our paper by a brief discussion of whether we need to consider a base extension from $k$ to $k'$ as in the conclusion of Theorem~\ref{main-theorem}; in particular, we prove in Proposition~\ref{prop:plane conics} that for a pencil of reduced plane conics (with at least one smooth conic in the $\overline{\F_q}$-pencil), there always exists a common transverse line to each element of the $\F_q$-pencil (as long as $q\ge 16$).

\

\textbf{Acknowledgements.} We are grateful to Zinovy Reichstein and Dori Bejleri for very helpful discussions on the topic of this paper.

\section{Plane curves}
\label{sec:plane}

In this Section, we discuss more broadly Question~\ref{quest:1} in the context of plane curves. In particular, we show (see Proposition~\ref{prop:N is not enough}) that given any $N$ reduced plane curves of degree $d$, there exists a common $\F_q$-line transverse to each one of these $N$ curves, as long as $q\ge 2Nd(d-1)$. Therefore, it makes sense to consider our Question~\ref{quest:1} in which we search for a common $\F_q$-line transverse to each curve in a given set of $q+1$ curves.  On the other hand, we show in Example~\ref{example:1} that there exists a set of $q+1$ smooth plane curves with the property that no $\F_q$-line is simultaneously transverse to each curve in our set. Hence, this suggests even more the setup considered in Question~\ref{quest:1}  in which we consider a \emph{pencil} of plane curves, or more generally of hypersurfaces in $\bP^n$.

The setup for this Section is to have two plane curves $C_{1}=\{F=0\}$ and $C_{2}=\{G=0\}$ in $\bP^2$ defined over $\F_q$. The polynomials $F, G\in \F_{q}[x,y,z]$ are homogenous of degree $d$, and we assume that $C_{1}\cap C_{2}$ is finite, i.e. the curves $C_1$ and $C_2$ do not share any components. We consider the pencil of plane curves,
$$
C_{[s:t]} = \{sF + tG = 0\} 
$$
We are interested in finding a line $L\subset\bP^2$ defined over $\F_q$ such that $L$ is simultaneously transverse to the $q+1$ members $C_{[s:t]}$ as $[s:t]$ varies in $\bP^1(\F_q)$. Note that a line $L\subset\mathbb{P}^2$ is transverse to a curve $C\subset\mathbb{P}^2$ if and only if $L\cap C$ consists of $d=\deg(C)$ distinct points (over $\overline{\mathbb{F}_q}$).

We need the following result on the number of $\F_q$-points to an arbitrary plane curve which is used in the proof of Proposition~\ref{transverse-line-single-curve}. 

\begin{lemma} \label{lemma-rat-points-on-plane-curves} Suppose $X\subset\bP^2$ is a plane curve of degree $d$ defined over $\F_q$. Then the number of $\F_q$-points of $X$ can be bounded by: 
$$
\# X(\F_q) \leq dq+1
$$
The equality occurs if $X$ is a union of $d$ lines, each defined over $\F_q$, passing through a common $\F_q$-point $P_0$.
\end{lemma}

\begin{proof} Note that if $d\geq q+1$, then $dq+1 \geq q^2+q+1 = \bP^2(\F_q)$, and the claim is trivially true. Thus, we may assume that $d<q+1$. First, we prove the result in the special case when $X$ has no $\F_q$-linear component. In this case, we prove a slightly stronger bound, namely $\# X(\F_q) \leq dq$. Consider the finite set,
$$
\mathcal{I} = \{ (P, L): P\in (X\cap L)(\F_q) \text{ and } L \text{ is an } \F_q\text{-line}\}.
$$
Given that each $P\in X(\F_q)$ is contained in exactly $q+1$ lines defined over $\F_q$, we get that $\# \mathcal{I} = (q+1)\cdot \#X(\F_q)$. On the other hand, using the assumption that $X$ contains no $\F_q$-line as a component, we deduce $L\cap X$ consists of at most $d$ $\F_q$-points by Bezout's theorem. Since the number of $\F_q$-lines is $q^2+q+1$, we obtain,
$$
\#\mathcal{I} \leq (q^2+q+1) d
$$
Combining the two inequalities, we get,
\[
(q+1) \cdot \# X(\F_q) \leq (q^2+q+1) d \ \ \Rightarrow \ \ \#X(\F_q) \leq \left(q+\frac{1}{q+1}\right)d < qd + 1
\]
where in the last step we used $d<q+1$. Thus, $\# X(\F_q) \leq qd$ for every plane curve $X$ which does not contain an $\F_q$-line as a component.

Now, suppose that $X$ contains an $\F_q$-line as a component. We induct on the degree of $X$ in this case. We write $X=L_0\cup Y$ where $L_0$ is an $\F_q$-line and $Y$ is a curve of degree $d-1$. If $Y$ does not contain an $\F_q$-line, then 
$$
\# X(\F_q) \leq \# L(\F_q) + \# Y(\F_q) \leq q+1 + (d-1)q = dq + 1
$$
as desired.  If $Y$ has an $\F_q$-line $L_1$, then by induction, $\# Y(\F_q) \leq (d-1)q+1$ but the point $P\colonequals L_0\cap L_1$ is counted twice, so 
$$
\# X(\F_q) \leq \# L(\F_q) + \# Y(\F_q) - 1 \leq q+1 + ((d-1)q + 1) - 1= dq + 1
$$
which completes the proof. 
\end{proof}

We note that Lemma~\ref{lemma-rat-points-on-plane-curves} is covered by a result of Serre \cite{Ser91} who proved a similar upper bound on the number of $\F_q$-points for an arbitrary projective hypersurface in $\bP^n$. Serre's result was generalized to all projective varieties by \cite{Cou16}.

\begin{proposition}\label{transverse-line-single-curve} Let $C\subset P^2$ be a reduced plane curve of degree $d$ defined over $\F_q$. If $q\geq 2d(d-1)$, then there exists a transverse $\F_q$-line to $C$. \end{proposition}

\begin{proof} Given a line $L=\{ax+by+cz=0\}\subset\bP^2$, we will show that the condition that $L$ is \emph{not} transverse to $C=\{F=0\}$ can be expressed in terms of vanishing of a certain discriminant. Indeed, we can solve for the intersection points 
$C\cap L$ by substituting $z = -(a/c) x - (b/c) y$ into the equation of $F(x,y,z)=0$ to obtain $F(x, y, -(a/c)x - (b/c)y)=0$. After homogenizing (which takes care of the possibility that $c$ could be $0$ in the above expression), the equation represents vanishing of a binary form $B_L(x,y)$ of degree $d$ in variables $x$ and $y$ with coefficients that are homogenous in variables $a, b, c$ with degree $d$. The line $L$ is non-transverse to $C$ if this binary form $B_L$ has a repeated root on $\bP^1$, i.e. the discriminant of $B_{L}$ vanishes. Since $\disc(B_{L})$ has degree $2d-2$ in the coefficients of the binary form, and the coefficients themselves are degree $d$ in variables $a, b, c$, we can view
$$
\disc(B_{L}) \in \F_{q}[a,b,c]
$$
as a homogenous form $H$ of degree $(2d-2)d = 2d(d-1)$ in variables $a, b, c$. By viewing a particular line $L$ as a point $[p:q:r]\in(\bP^2)^{\ast}$ in the dual space, we deduce that $L$ is tangent to $C$ if and only if the point $[p:q:r]$ lies on the plane curve $D=\{H=0\}$. In particular,
$$
\# \{ L\in (\bP^2)^{\ast}(\F_q) \ | \ L \text{ is a line not transverse to } C \} \leq \# D(\F_q)
$$
Since $D$ is a plane curve of degree $2d(d-1)$, the number of $\F_q$-points of $D$ can be bounded by $2d(d-1)q+1$ by Lemma \ref{lemma-rat-points-on-plane-curves}. Since the total number of  $\F_q$-lines in $\bP^2$ is $q^2+q+1$,  we will obtain a transverse $\F_q$-line to $C$ 
provided that
$$
q^2 + q + 1 > 2d(d-1)q+1
$$
This last inequality is equivalent to $q+1 > 2d(d-1)$, that is, $q\geq 2d(d-1)$. 
\end{proof} 

Using the same idea as in the previous proposition, we obtain:

\begin{proposition}
\label{prop:N is not enough} Let $C_1, C_2, ..., C_{N}$ be $N$ reduced plane curves of degree $d>1$ in $\bP^2$ defined over $\F_q$. If $q\geq 2N d(d-1)$, then there exists a common $\F_q$-line which is simultaneously transverse to $C_i$ for each $1\leq i\leq N$. \end{proposition}

\begin{proof} As in the proof of the previous proposition, we obtain that the number of non-transverse $\F_q$-lines to $C_i$ is at most $2d(d-1)q+1$. Thus, the number of lines that are non-transverse to at least one of the curves $C_1, C_2, ..., C_N$ is at most 
$N\cdot (2d(d-1)q+1)$. So, we will obtain a common transverse $\F_q$-line to all $C_i$ if 
$$
q^2 + q + 1 > N\cdot (2d(d-1)q+1)
$$
This inequality will be satisfied for $q \geq 2Nd(d-1)$ according to the following computation.
\begin{align*}
q^2 + q + 1 &= q(q+1)+1 \geq q(2Nd(d-1) + 1)+1 \\
&= 2Nd(d-1)q + q+1 > 2Nd(d-1)q + N = N\cdot (2d(d-1)q+1)
\end{align*}
where in the last inequality we used the fact that $q+1>N$ which is valid under the assumption $q \geq 2d(d-1)N$. 
\end{proof}

However, if the number of curves depend also on $q$, then the existence of a simultaneous transverse $\F_q$-line is not guaranteed. 

\begin{proposition}\label{example-singular-curves}
For each $d\geq 2$, there exist $q+1$ plane curves $C_1, C_2, ..., C_{q+1}$ of degree $d$ such that there is no $\F_q$-line which is transverse to each $C_i$. 
\end{proposition}

\begin{proof}
Fix an $\F_q$-line $L_0$ in $\bP^2$. After enumerating the $q+1$ $\F_q$-points $P_1, P_2, ..., P_{q+1}$ on $L_0=\bP^1$, construct the curve $C_i$ such that $C_i$ is \emph{any} given degree $d$ curve that is singular at the point $P_i$. The resulting collection of curves $C_1, ..., C_{q+1}$ satisfy the conclusion of the claim. Indeed, each $\F_q$-line $L$ meets $L_0$ at a unique point $P_i\in L_0$ (depending on $L$), and so $L$ passes through the singular point of $C_i$, implying that $L$ is not transverse to $C_i$. Thus, no $\F_q$-line $L$ can be simultaneously transverse to all the $q+1$ curves $C_{1}, C_{2}, ..., C_{q+1}$.
\end{proof}

It would be more satisfying to have examples of smooth curves satisfying the conclusion of Proposition~\ref{example-singular-curves}. We conjecture that such a collection of $q+1$ curves exist. 

\begin{conjecture}\label{conjecture-example-smooth-curves}
For each $d\geq 2$, there exist $q+1$  \emph{smooth} curves $C_1, C_2, ..., C_{q+1}$ in $\bP^2$ of degree $d$ such that there is no $\F_q$-line which is transverse to each $C_i$. 
\end{conjecture}

We can prove the conjecture in the special case when $d=2$. 

\begin{example}
\label{example:1} Suppose that the characteristic of the field is $p> 2$. We want to construct $q+1$ smooth conics $C_1, ..., C_{q+1}$ such that each $\F_q$-line $L$ in $\bP^2$ is tangent to at least one of $C_i$. The set of tangent lines to a given smooth conic $C$ is parametrized by the dual curve $C^{\ast}$ which also has degree $d(d-1)=2$. The condition that no $\F_q$-line is transverse to all of $C_1, ..., C_{q+1}$ can be translated into the statement that the $\F_q$-points of the corresponding dual curves $C_{1}^{\ast}, ..., C_{q+1}^{\ast}$ fill up all the $\F_q$-points of $(\bP^2)^{\ast}$. 

Motivated by the observation above, we proceed to construct $q+1$ smooth conics $D_{1}, D_{2}, ..., D_{q+1}$ such that 
$$
\bigcup_{i=1}^{q+1} D_i(\F_q) = \bP^2(\F_q)
$$
Consider the collection of $4$ points $\{P_1, P_2, P_3, P_4\} \subset \bP^2(\overline{\F_q})$ such that $\{P_1, P_2, P_3\}$ is a $\Gal(\F_{q^3}/\F_{q})$-orbit of the point $P_1\in\bP^2(\F_{q^3})$, while $P_4\in \bP^2(\F_q)$. In other words, if we write
$P_1 = [a: b: c]\in\bP^2(\F_{q^3})$, then $P_2 = [a^q : b^q: c^q]$ and $P_3 = [a^{q^2} : b^{q^2}: c^{q^2}]$. 

Furthermore, we can pick the collection $B \colonequals \{P_1, P_2, P_3, P_4\}$ in such a way that no three of $P_i$ are collinear. The vector space of homogeneous quadratic polynomials in 3 variables passing through $B$ has dimension $6-4=2$, and so we get a pencil of conics with base locus $B$. If $\{F_1, F_2\}$ is an $\F_q$-basis for this vector space, then we consider the $q+1$ members of the pencil,
$$
D_{[s, t]} \colonequals \{sF_1 + t F_2 = 0\}
$$
where $[s, t]\in\bP^1(\F_q)$. We claim that each $D_{[s:t]}$ is smooth. Indeed, there are only three singular conics (geometrically) in this pencil, and they are union of two lines passing through $B=\{P_1, P_2, P_3, P_4\}$. Using the notation $\overline{PQ}$ for the line passing through $P$ and $Q$, these 3 singular conics are:
\begin{align*}
S_1 &\colonequals \overline{P_1 P_2} \cup \overline{P_3 P_4} \\
S_2 &\colonequals \overline{P_2 P_3} \cup \overline{P_1 P_4} \\
S_3 &\colonequals \overline{P_1 P_3} \cup \overline{P_2 P_4}
\end{align*}
However, none of the $S_i$ for $1\leq i\leq 3$ is defined over $\F_q$. In fact, $S_1$ is strictly defined over the field $\F_{q^3}$, and Frobenius action sends $S_1 \to S_2 \to S_3 \to S_1$, and so $\{S_1, S_2, S_3\}$ is a Galois orbit of the Frobenius. In particular,  each $D_{[s:t]}$ is a smooth conic, and together they cover the $\F_q$-points of $\bP^2$. Indeed, on one hand, they all pass through $P_4\in \bP^2(\F_q)$; on the other hand, for each $P\in \bP^2(\F_q)\setminus \{P_4\}$, the conic $D_{[-F_2(P), F_1(P)]}$ passes through $P$. We re-label the elements of the pencil,
$$
\{ D_{[s:t]} \ | \ [s, t]\in \bP^1(\F_q) \} = \{D_1, D_2, ..., D_{q+1}\}
$$
So $D_1, ..., D_{q+1}$ are smooth conics which together cover the set $\bP^2(\F_q)$. Finally, we let $C_{i} = (D_{i})^{\ast}$ to be the corresponding dual curve for each $1\leq i \leq q+1$. By reflexivity, we have $D_{i} = (C_{i})^{\ast}$, and so the tangent lines to $C_{i}$ for $1\leq i\leq q+1$ together cover all the $\F_q$-lines of $\bP^2$, i.e. the collection of smooth conics $C_1, ..., C_{q+1}$ admit no common transverse $\F_q$-line.
\end{example}

\section{Main Result}
\label{sec:proofs}

In order to establish Theorem~\ref{main-theorem}, we will need the following lemma.

\begin{lemma}\label{lem:degree} Consider a pencil of hypersurfaces generated by $X_1$ and $X_2$ in $\bP^n$ defined over $k$.  Given a hyperplane $H\subset\bP^n$, either $H$ is non-transverse to every $\overline{k}$-member of the pencil, or $H$ is non-transverse to exactly $n(d-1)^{n-1}$ members of the pencil, counted with appropriate multiplicities.
\end{lemma}

\begin{proof} We have $X_1 = \{F_1 = 0\}$ and $X_2=\{F_2=0\}$ where $F_1, F_2\in \F_q[x_0, ..., x_n]$ are homogeneous polynomials of degree $d$. By definition, the elements of the pencil are of the form $X_{[s:t]} = \{sF_1 + t F_2 = 0\}$ as $[s:t]$ varies in $\bP^1$. Suppose that $H$ is an arbitrary hyperplane in $\bP^n$. After a linear change of coordinates, we may assume that $H=\{x_n=0\}$. We can restrict the original pencil to the hyperplane $H$ to obtain a new pencil whose elements are of the form,
$$
\tilde{X}_{[s:t]} = \{ s F_1(x_0, x_1, ..., x_{n-1}, 0) + t F_2 (x_0, x_1, ..., x_{n-1}, 0) = 0 \}
$$
which can be viewed as a pencil of hypersurfaces in $\bP^{n-1}$. Note that $H$ is transverse to $X_{[s:t]}$ if and only if $\tilde{X}_{[s:t]}=X_{[s: t]}\cap H$ is smooth. Thus, our task has been reduced to understanding how many of $\tilde{X}_{[s:t]}$ are singular. Let $\mathcal{D}_{d, n-1}$ be the discriminant hypersurface parametrizing singular hypersurfaces of degree $d$ in $\bP^{n-1}$, and $\mathcal{P}\cong \bP^1$ be the pencil whose members are $\tilde{X}_{[s:t]}$. Either $\mathcal{P}\subset \mathcal{D}_{d, n-1}$ or $\mathcal{P}\not\subset \mathcal{D}_{d, n-1}$. In the first case, $H$ is non-transverse to every member $X_{[s:t]}$ of the original pencil. In the second case, the number of the singular members of $\mathcal{P}$ is given by the degree of the discriminant $\mathcal{D}_{d, n-1}$, which is $n(d-1)^{n-1}$ according to \cite{EH16}*{Proposition 7.4}. Thus, $H$ is non-transverse to exactly $n(d-1)^{n-1}$ members of the original pencil, counted with multiplicity.  \end{proof}

We are now ready to present the proof of the main result.

\begin{proof}[Proof of Theorem~\ref{main-theorem}.]
We have a pencil of hypersurfaces generated by $X_1$ and $X_2$ such that the generic member of the pencil is smooth. Given $\zeta\in\bP^1$, we will denote by $X_{\zeta}$ to be the corresponding member of the pencil. Consider the variety,
$$
V = \{ (H, \zeta) \ | \ H \text{ is not transverse to } X_{\zeta} \} \subset (\bP^n)^{\ast} \times \bP^1
$$
We claim that $V$ is a geometrically irreducible variety. Consider the second projection $\pi_2\colon V\to\bP^1$. Since the generic member of the pencil is smooth, it follows that the generic fiber is irreducible. Indeed, if $X_{\zeta}$ is smooth, then the fiber 
$$
\pi_2^{-1}(\zeta) = \{ H \in (\bP^n)^{\ast} \ | \ H \text{ is tangent to } X_{\zeta} \} = (X_{\zeta})^{\ast}
$$
is the dual hypersurface, which is geometrically irreducible as it is the closure of the image of the irreducible hypersurface $X_{\zeta}$ under the Gauss map. Since $\pi_2: V\to\bP^1$ has geometrically irreducible fibers over an open set $U\subset \bP^1$ and $V$ is equidimensional (in fact, $V$ is a hypersurface because it can be seen as the dual hypersurface of the generic element of the pencil), it follows that $V$ is geometrically irreducible. 

Now, we consider the projection $\pi_1\colon V\to (\bP^n)^{\ast}$. Note that $\pi_1$ is surjective, because any chosen hyperplane is non-transverse to at least one element of the pencil by Lemma~\ref{lem:degree}. In fact, Lemma~\ref{lem:degree} shows 
that a fiber of $\pi_1$ either consists of $n(d-1)^{n-1}$ points (which is the generic case) or is an entire $\bP^1$. Let
$$
Z = \{P\in (\bP^n)^{\ast} \ | \ \pi_1^{-1}(P) = \bP^1\}
$$
consist of those hyperplanes $P$ that are simultaneously non-transverse to all the members of the pencil. In particular, such a hyperplane $P\in X_{1}^{\ast}\cap X_{2}^{\ast}$ for any two smooth members $X_1, X_2$ of the pencil. This shows that 
$Z\subset X_{1}^{\ast}\cap X_{2}^{\ast}$ and therefore $\dim(Z)\leq n-2$. In particular, $Z$ is a proper Zariski-closed subset in $(\bP^n)^{\ast}$. Since $V$ is geometrically irreducible, we can apply \cite{PS20}*{Theorem 1.8} to deduce that the locus 
$$
M_{\text{bad}} = \{ \text{hyperplanes } H \subset (\bP^n)^{\ast} \ | \ \pi_1^{-1}(H) \text{ is not geometrically irreducible}\}
$$
differs from a proper Zariski-closed subset by at most a constructible set of dimension $1$. As a result, $M_{\text{bad}}\neq (\bP^n)^{\ast}$. Thus, there exists a hyperplane $\mathcal{H} \hookrightarrow (\bP^n)^{\ast}$ such that $\mathcal{H}\notin M_{\text{bad}}$. Thus, we obtain a map $\pi_{1}\colon \pi_{1}^{-1}(\mathcal{H}) \to \mathcal{H}$. We apply \cite{PS20}*{Theorem 1.8} again to this new morphism, and continue inductively until we find a line $B = \bP^1 \subset \bP^{n-1}$ such that $W \colonequals \pi_{1}^{-1}(B)$ is a geometrically irreducible curve. Let $k_1/k$ be a finite field extension such that $B$ and $W$ are defined over $k_1$. We claim that $[k_1:k]$ depends only on $n$ and $d$. Indeed, $M_{\text{bad}}$ is a proper closed set whose degree and dimension are bounded by $n$ and $d$. Thus, Lang-Weil theorem ensures the existence of an $\F_q$-point in $(\bP^n)^{\ast}\setminus M_{\text{bad}}$ for $q$ sufficiently large with respect to $n$ and $d$. The same observation is true for each iteration of the inductive process, explaining why the degree $[k_1:k]$ depends only on $n$ and $d$.

We obtain a finite map $f: W \to B \cong\bP^1$ of geometrically irreducible curves over the field $k_1$; its degree is $m:=\deg(\pi_1)=n(d-1)^{n-1}$ by Lemma~\ref{lem:degree}, which is larger than $1$. Furthermore, the map is separable due to the hypothesis $p\nmid n(d-1)$. Note that $B\subset (\bP^{n})^{\ast}$, so a point $P\in B$ will correspond to a hyperplane $P$ in $\bP^n$. The fiber $f^{-1}(P)$ above a given point
$P\in B$ will be:
$$
f^{-1}(P) = \left\{ \zeta\in\bP^1 \ | \ P \text{ is non-transverse to } X_{\zeta} \right\}
$$
which is a finite set inside $\bP^1$.

Using the formulation above, we observe that a given $\F_q$-hyperplane $P\in B$ is simultaneously transverse to all the $\F_q$-members of the pencil generated by $X_1$ and $X_2$ if and only if the 
fiber $f^{-1}(P)$ contains no $\F_q$-points of $\bP^1$. In order to show the existence of such a point $P$, we will apply the Twisting Lemma of D\`ebes and Legrand \cite{DL12} to the cover $W/B$ after applying a suitable base extension. Note that $f: W\to B$ is a cover of geometrically irreducible curves; so, there exists a finite extension $k'/k_1$ such that the base extension of the cover $W_{k'}/B_{k'}$ has a regular Galois cover $Z_{k'}/B_{k'}$. More explicitly, $k'$ is the closure of $k_1$ inside the function field $Z(k_1)$. We also note that for any finite field $\F_q\supseteq k'$, it is still true that $Z_{\F_{q}}/B_{\F_{q}}$ is a regular Galois cover.

We claim that $k'/k$ depends only on $n$ and $d$. Indeed, $k'$ is the algebraic closure of $k_1$ inside $k_1(Z)$ and so, $[k':k_1]$ is bounded above by $[k_1(Z):k_1(B)]$ because $k_1(B)$ is the rational function field over $k_1$ (since $B$ is isomorphic to $\bP^1$) and so, $k_1$ is closed inside $k_1(B)$. Moreover, $Z/B$ is the Galois closure of $W/B$. As $W/B$ has degree $n(d-1)^{n-1}$, it follows that $Z/B$ has degree bounded above by $(n(d-1)^{n-1})!$. We deduce that $[k':k_1]$ is uniformly bounded solely in terms of $n$ and $d$. This shows that the extension $k'/k_1$ and therefore also $k'/k$ depends only on $n$ and $d$.

For the rest of the proof, let $\F_q\supseteq k'$ be any finite field. Let $G$ be the Galois group of $Z_{\F_q}/B_{\F_q}$; we view $G$ as a subgroup of $S_m$.

We will apply \cite{DL12}*{Lemma~3.4} to the map $f: W_{\F_q} \to B_{\F_q}$ in order to obtain a point $P\in B(\F_q)$ with the property that no point in $f^{-1}(P)$ is contained in $W(\F_q)$. 

We need first a cyclic subgroup $H$ of $G$ generated by an element $\sigma\in S_m$ with the property that $\sigma$ fixes no element in $\{1,\dots, m\}$ (note that $m>1$). Indeed, for any Galois group $G$ (seen as a subgroup of $S_m$), there exists an element $\sigma\in G$ which has no fixed point in $\{1,\dots, m\}$ because $G$ is a transitive group, which means that the stabilizers of the elements in $\{1,\dots,m\}$ are all conjugated and finally,
no group is a union of conjugates of a given proper subgroup.

So, we let $H$ be a cyclic subgroup of $G$ generated by an element $\sigma$ which has no fixed points (as above); we let $r$ be the number of all cycles appearing in $\sigma\in S_m$. We consider the \'etale $\F_{q}$-algebra $\prod_{\ell=1}^r E_\ell$, where the $E_\ell$'s are field extensions of $\F_{q}$ of degrees equal to the orders of the cycles appearing in the permutation $\sigma$. Then we apply \cite{DL12}*{Lemma~3.4} to the \'etale algebra $\prod_{\ell=1}^r E_\ell/\F_{q}$ to obtain a point $P\in B(\F_{q})$ with the property that $f^{-1}(P)$ splits into $r$ Galois orbits of order $[E_\ell:\F_{q}]$; in particular, none of the points in $f^{-1}(P)$ would be contained in $W(\F_{q})$ since each of these Galois orbits would have cardinality larger than $1$ (because $\sigma$ does not have fixed points).  

Now, the hypothesis in applying \cite{DL12}*{Lemma~3.4} is satisfied because the (const/comp) condition from \cite{DL12}*{Section~3.1.1} is automatically satisfied for regular covers. We need to check the following two conditions, namely \cite{DL12}*{Lemma 3.4, conditions (ii)-1 and (ii)-2}:
\begin{enumerate}
\item This condition is automatically satisfied for large $q$, because the Lang-Weil bounds for the number of points of curves defined over finite fields guarantees the existence of many rational points on the corresponding twisted covers of $Z$, which are curves of the same genus as the genus of $Z$ (see also the proof of \cite{DL12}*{Corollary~4.3}). Note that $q$ can be made to be sufficiently large by extending the field $k'$ even further in a way so that $[k':k]$ would still only depend on $n$ and $d$; indeed, Lang-Weil bounds apply once $q$ is larger than some function of the genus of $Z$. Since $Z$ is a degree $\delta$ cover of $\bP^1$, where $\delta$ is bounded above solely in terms of $d$ and $n$, it follows that the genus of $Z$ is also bounded solely in terms of $d$ and $n$.

\item This condition is satisfied as explained in the discussion regarding cyclic specializations (since our group $H$ is cyclic) on \cite{DL12}*{p.~153}.
\end{enumerate}
Therefore, \cite{DL12}*{Lemma~3.4} yields the existence of a point $P\in B(\F_{q})$ such that no point in $f^{-1}(P)$ is contained in $W(\F_{q})$, concluding the proof of Theorem~\ref{main-theorem}. 
\end{proof}

\begin{remark}
\label{rem:s larger than 1}
In our proof of Theorem~\ref{main-theorem} we used that the ground field $k$ may have to be replaced by $k'$ when considering the Galois closure $Z/B$ for the cover $W/B$ since we want that $Z$ be geometrically irreducible (over $k'$). Note that there are covers of degree larger than $1$ of geometrically irreducible curves $W/B$ (over $k$) for which each $k$-point of $B$ has a preimage contained in $W(k)$, thus contradicting the conclusion we seek for the strategy of our proof of Theorem~\ref{main-theorem}.

Indeed, we let $k=\F_q$ and $W=B=\bP^1_{\F_q}$ for some prime power $q$ satisfying the congruence equation $q\equiv 2\pmod{3}$ and then let $f:\bP^1\longrightarrow \bP^1$ be given by $x\mapsto x^3$. Clearly, $f$ induces a permutation of $\bP^1_{\F_q}$; so, each point in $B(\F_q)$ has a preimage contained in $W(\F_q)$. On the other hand, the Galois closure of this cover is $Z=\bP^1_{\F_{q^2}}$, i.e., we need to perform a base extension of our ground field in order for the Galois cover be geometrically irreducible. Once we replace $q$ by $q^2$, then $W_{\F_{q^2}}/B_{\F_{q^2}}$ is actually a regular Galois cover and then it is true that there exist points $P\in B(\F_{q^2})$ such that no point in $f^{-1}(P)$ is contained in $W(\F_{q^2})$.  
\end{remark}

We do not know whether one can choose $k'=k$ in Theorem~\ref{main-theorem} in general, as our proof strategy requires a base extension (see Remark~\ref{rem:s larger than 1}). It might be reasonable to expect that if the cardinality of the ground field $k$ is sufficiently large (depending only on $n$ and $d$), then one does not require an additional field extension. For example, the following result establishes that $k'=k$ works for the case of pencil of plane conics (as long as $\#k \ge 16$).

\begin{proposition}
\label{prop:plane conics}
Suppose that we have a pencil of reduced conics in $\bP^2$ defined over $\F_q$ such that the pencil admits at least one smooth member over $\overline{\F_q}$. Provided that $q\geq 16$, we can find an $\F_q$-line $L$ that is simultaneously transverse to all the conics defined over $\F_q$
in the pencil.
\end{proposition}

\begin{proof}
Suppose that $C_1 = \{F_1=0\}$ and $C_2=\{F_2=0\}$ are the two conics that generate the pencil. 

We start with some general considerations regarding our proof strategy. First, we observe that if $C$ is a non-smooth reduced conic, then it means that $C$ is a union of two lines $L_1\cup L_2$ (over $\overline{\F_q}$) and therefore, we have at most $q+1$ lines defined over $\F_q$ which are non-transverse to $C$ (they would correspond to all the $\F_q$-lines passing through the $\F_q$-point of $L_1\cap L_2$).  Second, we note that if $C$ is any smooth conic defined over $\F_q$, then the only possibility for an $\F_q$-line $L$ be non-transverse to $C$ is for $L$ be tangent to $C$ at an $\F_q$-point (since otherwise, we would have that $L$ is tangent to $C$ at two $\overline{\F_q}$-points, contradiction). In particular, if $C$ is a smooth conic which has no $\F_q$-point, then any $\F_q$-line is transverse to $C$. On the other hand, the number of $\F_q$-points on a smooth $\F_q$-conic (which has at least one $\F_q$-point) is $q+1$ (since then the conic would be isomorphic to $\bP^1$ over $\F_q$); furthermore, each such $\F_q$-point has a tangent line defined over $\F_q$. This provides at most $(q+1)\cdot (q+1)$ lines defined over $\F_q$, which are non-transverse to at least one element of the given $\F_q$-pencil. This number is an overestimate since there are only $q^2+q+1$ lines defined over $\F_q$, and so there is overcounting that needs to be addressed. In order to refine the counting for the number of non-transverse $\F_q$-lines, we need to take into account the fact that a given $\F_q$-line $L$ will be non-transverse to more than one conic.

In the set-up of the proof for the Theorem~\ref{main-theorem}, we have the map $\pi_1: V\to (\bP^2)^{\ast}$. Given a line $L\in (\bP^2)^{\ast}$, the fiber $\pi_1^{-1}(L)$ is either a $\bP^1$ or consists of $2$ conics according to Lemma~\ref{lem:degree}. In the first case, the line $L$ is non-transverse to every element of pencil, and in the second case $L$ is non-transverse to exactly $2$ conics (counted with multiplicity). In most cases, we see that each non-transverse $\F_q$-line is counted at least twice. However,  there is a locus $\mathcal{B}\subset (\bP^2)^{\ast}$ consisting of those lines $L\in (\bP^2)^{\ast}$ which are tangent to exactly one conic (with multiplicity 2) in the pencil. We claim that $\mathcal{B}$ is a plane curve of degree $4$. 

The variety $V \subset \bP^1 \times (\bP^2)^{\ast}$ can be described as the locus $\{R(s, t, a, b, c) = 0\}$ which has bidegree $(2, 2)$, that is, degree $2$ in variables $s, t$ and degree $2$ in variables $a, b, c$.  The two roots $[s:t]\in \bP^1$ satisfying $R(s, t, a, b, c)=0$ exactly correspond to those members of the pencil to which a given line $L=\{ax+by+cz=0\}$ is non-transverse. The condition that these two roots coincide is controlled by the vanishing of the discriminant $D$ of $R(s,t,a,b,c)$ when $R$ is viewed as a homogeneous quadratic polynomial in $s$ and $t$. Note that $D=D(a,b,c)$ is a degree $4$ homogeneous polynomial in $a, b, c$. By definition, $\mathcal{B}=\{D=0\}$ and so $\deg(\mathcal{B})=4$. By Lemma~\ref{lemma-rat-points-on-plane-curves}, we have $\# \mathcal{B}(\F_q)\leq 4q+1$, and so there are at most $4q+1$ lines over $\F_q$ which are non-transverse to a single conic (with multiplicity $2$) in the pencil.

Finally, there are at most three distinct singular conics in a given pencil of conics by  \cite{EH16}*{Proposition 7.4}. Each such conic is a union of two lines, and the only lines that are not transverse are the $\F_q$-lines passing through the singular point. Thus, there are at most $3(q+1)$ non-transverse lines arising from the singular conics in the pencil.

In total, the number of non-transverse $\F_q$-lines to the $\F_q$-members of the pencil is at most $\frac{(q+1)^2}{2}  + 4q+1 + 3(q+1)$. Since the number of $\F_q$-lines is $q^2+q+1$, we get a simultaneously transverse $\F_q$-line provided that,
$$
q^2 + q + 1 > \frac{(q+1)^2}{2}  + 4q+1 + 3(q+1)
$$
The inequality above is equivalent to $q^2>14q+7$ which is true for $q\geq 16$.
\end{proof}


\begin{bibdiv}
\begin{biblist}

\bib{Asg19}{article}{
    AUTHOR = {Asgarli, Shamil},
     TITLE = {Sharp {B}ertini theorem for plane curves over finite fields},
   JOURNAL = {Canad. Math. Bull.},
  FJOURNAL = {Canadian Mathematical Bulletin. Bulletin Canadien de Math\'{e}matiques},
    VOLUME = {62},
      YEAR = {2019},
    NUMBER = {2},
     PAGES = {223--230},
      ISSN = {0008-4395},
   MRCLASS = {14H50 (11G20 14N05)},
  MRNUMBER = {3952512},
MRREVIEWER = {Mariana de Almeida Nery Coutinho},
       DOI = {10.4153/cmb-2018-018-0},
       URL = {https://doi-org.ezproxy.library.ubc.ca/10.4153/cmb-2018-018-0},
}

\bib{ADL19}{article}{
       author = {Asgarli, Shamil},
       author ={Duan, Lian},
       author = {Lai, Kuan-Wen},
        title = {Transverse lines to surfaces over finite fields},
      journal = {arXiv e-prints},
      FJOURNAL = {},
     keywords = {Mathematics - Algebraic Geometry},
         year = {2019},
        month = {mar},
          eid = {arXiv:1903.08845},
        pages = {arXiv:1903.08845},
      archivePrefix = {arXiv},
       eprint = {https://arxiv.org/abs/1903.08845},
     primaryClass = {math.AG},
       adsurl = {https://ui.adsabs.harvard.edu/abs/2019arXiv190308845A},
      adsnote = {Provided by the SAO/NASA Astrophysics Data System}
}

\bib{Bal03}{article}{
    AUTHOR = {Ballico, Edoardo},
     TITLE = {An effective {B}ertini theorem over finite fields},
   JOURNAL = {Adv. Geom.},
    VOLUME = {3},
      YEAR = {2003},
    NUMBER = {4},
     PAGES = {361--363},
      ISSN = {1615-715X}
}

\bib{Cou16}{article}{
    AUTHOR = {Couvreur, Alain},
     TITLE = {An upper bound on the number of rational points of arbitrary
              projective varieties over finite fields},
   JOURNAL = {Proc. Amer. Math. Soc.},
  FJOURNAL = {Proceedings of the American Mathematical Society},
    VOLUME = {144},
      YEAR = {2016},
    NUMBER = {9},
     PAGES = {3671--3685},
      ISSN = {0002-9939},
   MRCLASS = {11G25 (14J20)},
  MRNUMBER = {3513530},
MRREVIEWER = {Mrinmoy Datta},
       DOI = {10.1090/proc/13015},
       URL = {https://doi-org.ezproxy.library.ubc.ca/10.1090/proc/13015},
}

\bib{DL12}{incollection}{
    AUTHOR = {D\`ebes, Pierre},
    AUTHOR = {Legrand, Fran\c{c}ois},
     TITLE = {Twisted covers and specializations},
 BOOKTITLE = {Galois-{T}eichm\"{u}ller theory and arithmetic geometry},
    SERIES = {Adv. Stud. Pure Math.},
    VOLUME = {63},
     PAGES = {141--162},
 PUBLISHER = {Math. Soc. Japan, Tokyo},
      YEAR = {2012},
   MRCLASS = {12E30 (11R58 14H30)},
  MRNUMBER = {3051242},
MRREVIEWER = {Teresa Crespo},
       DOI = {10.2969/aspm/06310141},
       URL = {https://doi-org.ezproxy.library.ubc.ca/10.2969/aspm/06310141},
}

\bib{EH16}{book}{
    AUTHOR = {Eisenbud, David},
    AUTHOR ={Harris, Joe},
     TITLE = {3264 and all that---a second course in algebraic geometry},
 PUBLISHER = {Cambridge University Press, Cambridge},
      YEAR = {2016},
     PAGES = {xiv+616},
      ISBN = {978-1-107-60272-4; 978-1-107-01708-5},
   MRCLASS = {14-01 (14C15 14M15 14N10)},
  MRNUMBER = {3617981},
MRREVIEWER = {Arnaud Beauville},
       DOI = {10.1017/CBO9781139062046},
       URL = {https://doi-org.ezproxy.library.ubc.ca/10.1017/CBO9781139062046},
}

\bib{Kat99}{article}{
    AUTHOR = {Katz, Nicholas M.},
     TITLE = {Space filling curves over finite fields},
   JOURNAL = {Math. Res. Lett.},
  FJOURNAL = {Mathematical Research Letters},
    VOLUME = {6},
      YEAR = {1999},
    NUMBER = {5-6},
     PAGES = {613--624},
      ISSN = {1073-2780},
   MRCLASS = {11G20 (11G10 11G25 14G15)},
  MRNUMBER = {1739219},
MRREVIEWER = {Arnaldo L. P. Garc\'{\i}a},
       DOI = {10.4310/MRL.1999.v6.n6.a2},
       URL = {https://doi-org.ezproxy.library.ubc.ca/10.4310/MRL.1999.v6.n6.a2},
}

\bib{PS20}{article}{
       author = {Poonen, Bjorn},
       author = {Slavov, Kaloyan},
        title = {The exceptional locus in the Bertini irreducibility theorem for a morphism},
      journal = {arXiv e-prints},
     keywords = {Mathematics - Algebraic Geometry, Mathematics - Number Theory, 14D05 (Primary) 14A10, 14G15 (Secondary)},
         year = {2020},
        month = {jan},
          eid = {arXiv:2001.08672},
        pages = {arXiv:2001.08672},
archivePrefix = {arXiv},
       eprint = {https://arxiv.org/abs/2001.08672},
 primaryClass = {math.AG},
       adsurl = {https://ui.adsabs.harvard.edu/abs/2020arXiv200108672P},
      adsnote = {Provided by the SAO/NASA Astrophysics Data System}
}

\bib{Ser91}{incollection}{
    AUTHOR = {Serre, Jean-Pierre},
     TITLE = {Lettre \`a {M}. {T}sfasman},
      NOTE = {Journ\'{e}es Arithm\'{e}tiques, 1989 (Luminy, 1989)},
   JOURNAL = {Ast\'{e}risque},
  FJOURNAL = {Ast\'{e}risque},
    NUMBER = {198-200},
      YEAR = {1991},
     PAGES = {11, 351--353 (1992)},
      ISSN = {0303-1179},
   MRCLASS = {14G15 (11G25 14G05)},
  MRNUMBER = {1144337},
MRREVIEWER = {Paulo Viana},
}

\end{biblist}
\end{bibdiv}
\end{document}